\documentclass{amsart}
\usepackage{hyperref}
\usepackage{amsfonts}
\usepackage{amsmath}
\usepackage{amssymb}
\usepackage{amscd}
\usepackage{graphics}
\usepackage{latexsym}
\usepackage{color}
\usepackage{epsfig}
\usepackage{amsopn}

\newtheorem{theorem}{Theorem}[section]
\newtheorem{proposition}[theorem]{Proposition}
\newtheorem{lemma}[theorem]{Lemma}
\newtheorem{corollary}[theorem]{Corollary}

\theoremstyle{remark}
\newtheorem{remark}[theorem]{Remark}

\newtheorem{example}[theorem]{Example}

\newcommand{\PP}{\mathbb{P}}

\newcommand{\OO}{\mathcal{O}}
\newcommand{\OOO}{\mathbb{O}}

\newcommand{\ZZ}{\mathbb{Z}}

\def\tmax{\mathrm{max}}\def\tmin{\mathrm{min}}
\def\op{\oplus}\def\cT{\mathcal{T}}
\def\fg{\mathfrak{g}}\def\fr{\mathfrak{r}}
\def\fq{\mathfrak{q}}\def\fp{\mathfrak{p}}
\def\bP{\mathbb{P}} \def\bbO{\mathbb{O}}
\def\cayley{\mathbb{OP}^2} \def\freud{G(\bbO^3,\bbO^6)}
\def\cQ{\mathcal{Q}}
\newcounter{acnt}
\newenvironment{a_list}{ 
  \begin{list}{{(\alph{acnt})}}
   {\usecounter{acnt} \setlength{\itemsep}{3pt}
    \setlength{\leftmargin}{20pt} \setlength{\labelwidth}{20pt} }
   }
   {\end{list}}

\begin{document}

\title{Flexibility of Schubert classes}

\author{Izzet Coskun} 

\date{} 

\address{University of Illinois at
  Chicago, Department of Mathematics, Statistics and Computer Science, Chicago, IL 60607}

\email{coskun@math.uic.edu}

\author{Colleen Robles}
\address{Texas A \& M University, Mathematics Department, College Station, TX 77843}
\email{robles@math.tamu.edu}

\thanks{During the preparation of this article the author was
  partially supported by the NSF CAREER grant DMS-0950951535 and an Alfred P. Sloan Foundation Fellowship, the second author was partially supported by the NSF grant DMS-1006353.}

\begin{abstract}
In this note, we discuss the flexibility of Schubert classes in homogeneous varieties. We give several constructions for representing multiples of a Schubert class by irreducible subvarieties. We sharpen  \cite[Theorem 3.1]{robles2}  by proving that every positive multiple of an obstructed class in a cominuscule homogeneous variety can be represented by an irreducible subvariety.  
\end{abstract}


\maketitle

\begin{figure}[p] 
\setlength{\unitlength}{25pt}
\begin{picture}(5,17)
\put(5,0){\line(-1,1){5}}\put(3,4){\line(-1,1){2}}
\put(3,6){\line(-1,1){2}}\put(4,7){\line(-1,1){4}}
\put(5,8){\line(-1,1){4}}\put(3,12){\line(-1,1){1}}
\put(0,11){\line(1,1){5}}\put(1,10){\line(1,1){2}}
\put(1,8){\line(1,1){2}}\put(0,5){\line(1,1){4}}
\put(1,4){\line(1,1){4}}\put(2,3){\line(1,1){1}}
\multiput(5,0)(-1,1){6}{\circle*{0.15}}
\multiput(3,4)(-1,1){3}{\circle*{0.15}}
\multiput(3,6)(-1,1){3}{\circle*{0.15}}
\multiput(4,7)(-1,1){5}{\circle*{0.15}}
\multiput(5,8)(-1,1){5}{\circle*{0.15}}
\multiput(3,12)(-1,1){2}{\circle*{0.15}}
\multiput(5,16)(-1,-1){3}{\circle*{0.15}}
\thicklines
\put(3,4){\circle{0.28}} \put(0,5){\circle{0.28}}
\put(1,8){\circle{0.28}} \put(5,8){\circle{0.28}}
\put(0,11){\circle{0.28}} \put(3,12){\circle{0.28}}
\put(5,0){\circle{0.28}} \put(5,16){\circle{0.28}}
\thinlines
\put(4.2,14.9){78 +} \put(3.2,13.9){78 +} \put(2.2,12.9){78 +}
\put(3.2,11.9){33} \put(1.2,11.9){45 T}
\put(0.2,10.9){12} \put(-2.4,10.9){\small{$C(OG(5,10))$}} 
\put(2.2,10.9){33 +}\put(1.2,9.9){12 +} 
\put(3.2,9.9){21 T} \put(2.2,8.9){12 +} \put(4.2,8.9){9 T}
\put(1.2,7.9){5} \put(3.2,7.9){7 *} 
\put(5.2,7.9){2} \put(5.0,8.3){$\mathcal{Q}^8$} 
\put(2.2,6.9){5 +} \put(4.2,6.9){2 +}
\put(1.2,5.9){3 *} \put(3.2,5.9){2 +}
\put(0.2,4.9){1} \put(-0.1,5.25){$\PP^5$} 
\put(2.2,4.9){2 +}
\put(1.2,3.9){1 +} \put(3.2,3.9){1} \put(3.1,4.2){$\PP^4$} 
\put(2.2,2.9){1 +} \put(3.2,1.9){1 +} \put(4.2,0.9){1 +}
\put(5.2,-0.1){1} \put(5.2,15.9){78}
\normalsize
\put(3.7,-0.1){$o \in X$} \put(4.4,15.8){$X$}
\end{picture}
\caption{Hasse diagram of the Cayley plane $\mathbb{O}\PP^2$.}
\label{f:E6} 
\end{figure}
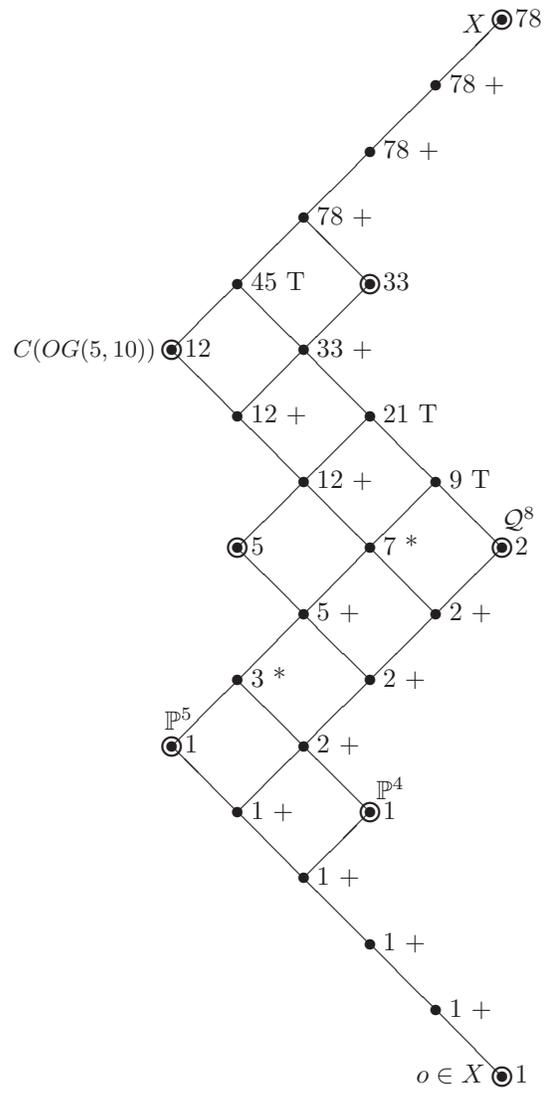
\begin{figure}[p]
\setlength{\unitlength}{20pt}
\begin{picture}(8,28)
\put(2,6){\line(1,-1){6}} \put(3,7){\line(1,-1){2}}
\put(3,9){\line(1,-1){2}} \put(2,12){\line(1,-1){4}}
\put(3,13){\line(1,-1){4}} \put(4,14){\line(1,-1){4}}
\put(0,17){\line(1,-1){4}} \put(1,18){\line(1,-1){4}}
\put(2,19){\line(1,-1){4}} \put(3,20){\line(1,-1){2}}
\put(3,22){\line(1,-1){2}} \put(0,27){\line(1,-1){6}}
\put(4,4){\line(1,1){1}} \put(3,5){\line(1,1){5}}
\put(2,6){\line(1,1){5}} \put(3,9){\line(1,1){3}}
\put(3,11){\line(1,1){2}} \put(2,12){\line(1,1){2}}
\put(4,13){\line(1,1){2}} \put(3,14){\line(1,1){2}}
\put(4,12){\line(0,1){1}} \put(4,14){\line(0,1){1}}
\put(3,13){\line(0,1){1}} \put(5,13){\line(0,1){1}}
\put(2,15){\line(1,1){3}} \put(1,16){\line(1,1){5}}
\put(0,17){\line(1,1){5}} \put(3,22){\line(1,1){1}}
\multiput(2,6)(1,-1){7}{\circle*{0.25}} \multiput(3,7)(1,-1){3}{\circle*{0.25}}
\multiput(3,9)(1,-1){3}{\circle*{0.25}} \multiput(2,12)(1,-1){5}{\circle*{0.25}}
\multiput(3,13)(1,-1){5}{\circle*{0.25}} \multiput(4,14)(1,-1){5}{\circle*{0.25}}
\multiput(0,17)(1,-1){5}{\circle*{0.25}} \multiput(1,18)(1,-1){5}{\circle*{0.25}}
\multiput(2,19)(1,-1){5}{\circle*{0.25}} \multiput(3,20)(1,-1){3}{\circle*{0.25}}
\multiput(3,22)(1,-1){3}{\circle*{0.25}} \multiput(0,27)(1,-1){7}{\circle*{0.25}}
\thicklines
\put(0,27){\circle{0.45}}\put(8,0){\circle{0.45}}
\put(5,5){\circle{0.45}} \put(2,6){\circle{0.45}} 
\put(3,9){\circle{0.45}} \put(8,10){\circle{0.45}} 
\put(2,12){\circle{0.45}} \put(4,13){\circle{0.45}} 
\put(4,14){\circle{0.45}} \put(6,15){\circle{0.45}} 
\put(0,17){\circle{0.45}} \put(6,21){\circle{0.45}} 
\put(3,22){\circle{0.45}} \put(5,18){\circle{0.45}} 
\thinlines
\footnotesize
\put(6.5,-0.2){$o \in X$} \put(-0.7,26.9){$X$}
\put(5.4,5.2){$\PP^5$} \put(1.3,6.2){$\PP^6$}
\put(8.7,9.9){$\mathcal{Q}^{10}$}
\put(8.4,-0.15){1}
\put(7.2,0.85){1 +} \put(6.2,1.85){1 +} \put(5.2,2.85){1 +}
\put(4.2,3.85){1 +} \put(5.4,4.8){1} \put(3.3,4.85){1 +}
\put(4.2,5.85){2 +} \put(2.4,5.85){1}
\put(5.2,6.85){2 +} \put(3.2,6.85){3 *}
\put(4.2,7.85){5 +} \put(6.2,7.85){2 +}
\put(3.4,8.85){5} \put(5.3,8.85){7 *} \put(7.2,8.85){2 +}
\put(4.2,9.85){12 +} \put(6.2,9.85){9 *} \put(8.4,9.9){2}
\put(3.2,10.85){12 +} \put(5.2,10.85){21 *} \put(7.2,10.85){11 T}
\put(2.4,11.85){12} \put(4.2,11.85){33 +} \put(6.2,11.85){32 T}
\put(2,12.85){45 *} \put(4.0,12.4){33} \put(5.2,12.85){65 T}
\put(2.0,13.85){78 +} \put(3.75,14.3){110} \put(5.2,13.85){98 +}
\put(2.2,14.85){78 +} \put(4.2,14.85){286 T} \put(6.4,14.85){98}
\put(1.2,15.85){78 +} \put(3.2,15.85){364 T} \put(5.2,15.85){384 T}
\put(0.4,16.85){78} \put(-1.9,16.9){$C(\mathbb{O}\PP^2)$}
\put(2.2,16.85){442 T} \put(4.2,16.85){748 +}
\put(1.2,17.85){520 T} \put(3.3,17.85){1190 T} \put(5.4,17.85){748}
\put(2.2,18.85){1710 T} \put(4.2,18.85){1938 +} 
\put(3.2,19.85){3648 T} \put(5.2,19.85){1938 +}
\put(4.2,20.85){5586 +} \put(6.4,20.85){1938}
\put(3.4,21.85){5586} \put(5.3,21.85){7524 T}
\put(4.2,22.85){13110 +} \put(3.2,23.85){13110 +} \put(2.2,24.85){13110 +}
\put(1.2,25.85){13110 +} \put(0.3,26.85){13110}
\normalsize
\end{picture}
\caption{Hasse diagram of the Freudenthal variety $G(\mathbb{O}^3,\mathbb{O}^6)$.}
\label{f:E7} 
\end{figure}
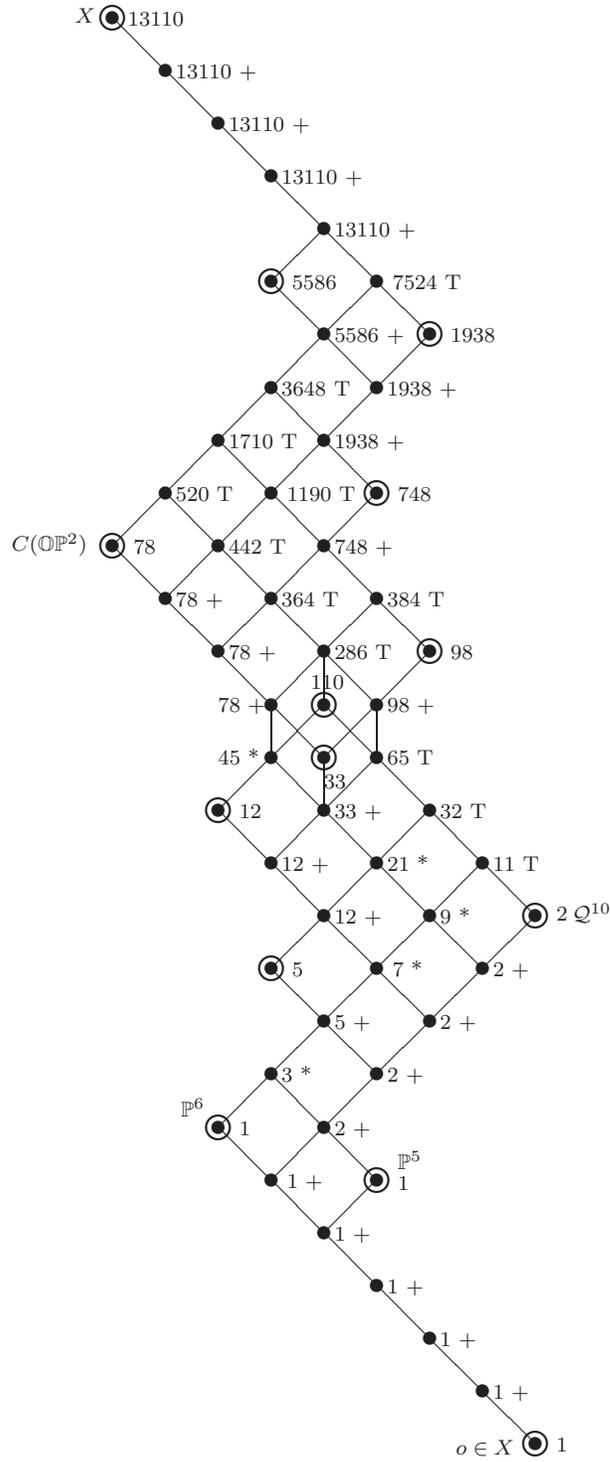

\section{Introduction}

Let $G$ be a linear algebraic group and let $P$ be a parabolic subgroup. The cohomology of the homogeneous variety $X= G/P$ admits a $\ZZ$-basis given by the classes of Schubert varieties. In 1961, Borel and Haefliger \cite{borel} asked whether a Schubert class $\sigma$ can be represented by projective subvarieties of $X$ other than Schubert varieties. A Schubert class $\sigma$ is called {\em rigid} if Schubert varieties are the only projective subvarieties representing $\sigma$. More generally, given a positive integer $m$, the class $m \sigma$ can be represented by a union of $m$ Schubert varieties. The class $\sigma$ is called {\em multi rigid} if these are the only projective algebraic subsets  representing $m \sigma$ for every positive integer $m$. Otherwise, the class $\sigma$ is called {\em flexible}.

The rigidity of Schubert classes has been studied by many authors (see  \cite{bryant:rigidity}, \cite{coskun:rigid}, \cite{coskun:orthogonal}, \cite{coskun:symplectic}, \cite{hartshorne:smooth},\cite{hong:rigidity}, \cite{hong:singular},  \cite{robles2}, \cite{robles} and \cite{walters:thesis}). Robles and The \cite{robles}, \cite{robles2}, building on the work of R. Bryant \cite{bryant:rigidity}, J. Hong \cite{hong:rigidity}, \cite{hong:singular}, and M. Walters \cite{walters:thesis},  characterized the multi rigid Schubert classes in cominuscule homogeneous varieties. First, in \cite{robles}, the authors identified a space $\mathcal{O}_\sigma$ of first order obstructions to the multi rigidity 
of a Schubert class $\sigma$ in Lie algebra cohomology: if these obstructions vanish, $\mathcal{O}_\sigma=0$, then $\sigma$ is multi rigid. We call a Schubert class {\em obstructed} if $\mathcal{O}_\sigma \not= 0$. Then Robles \cite{robles2} completed the classification by constructing an irreducible variety representing a multiple of $\sigma$ (in almost all cases $2 \sigma$ or $4 \sigma$) for every obstructed class in cominuscule homogeneous varieties.

In this note, we give several simple constructions of irreducible subvarieties representing multiples of Schubert classes.  Our main result is the following sharpening of \cite[Theorem 3.1]{robles2}.  

\begin{theorem}\label{main-theorem}
Let $\sigma$ be an obstructed Schubert class in the cohomology of a cominuscule homogeneous variety $X$ and let $m$ be a positive integer. Then there exists an irreducible subvariety of $X$ representing $m \sigma$. In particular, obstructed Schubert classes in $H^*(X, \ZZ)$ are flexible. 
\end{theorem}

As a consequence of Theorem \ref{main-theorem}, Schubert classes in cominuscule homogeneous varieties satisfy a dichotomy. Either every positive multiple of $\sigma$ can be represented by an irreducible subvariety, or none of the multiples bigger than one can be represented by an irreducible subvariety.  We pose the question whether the same dichotomy holds for Schubert classes in any homogeneous variety. Our constructions apply to all homogeneous varieties and show that a large set of Schubert classes in arbitrary homogeneous varieties also satisfy the dichotomy. However, at present, we do not have a complete classification of multi rigid Schubert classes beyond the cominuscule case. 

There are several standard ways of constructing irreducible subvarieties of a (homogeneous) variety $X$.  First, one may intersect an irreducible subvariety $Y$ of $X$ with a general hypersurface $Z$ of degree $d$ in the minimal embedding of $X$. If $\dim(Y) \geq 2$, then, by Bertini's Theorem, $Y \cap Z$ is also irreducible. Hence, when the class of $Y \cap Z$ is $m \sigma$, we can represent $m \sigma$ by an irreducible subvariety.

Second, homogeneous varieties represent functors. For example, the Grassmannian $G(k,n)$ parameterizes $k$-dimensional subspaces of an $n$-dimensional vector space. One can define subvarieties of a homogeneous variety $X$ by imposing geometric conditions, such as incidence and tangency conditions, on the objects parameterized by $X$. For example, Schubert varieties themselves are defined by imposing incidence conditions with respect to a fixed flag. 

Third, one may use correspondences to obtain irreducible subvarieties of $X$. Inclusions among homogeneous varieties and Tits correspondences are the most fruitful for constructing irreducible subvarieties representing multiples of Schubert classes. These three standard constructions suffice to prove Theorem \ref{main-theorem}. 

The organization of this paper is as follows. In \S \ref{sec-preliminary}, we will recall basic facts concerning rigidity of Schubert classes in cominuscule homogeneous spaces. In \S \ref{sec-proof}, we will construct irreducible subvarieties representing multiples of obstructed Schubert classes in cominuscule varieties and prove Theorem \ref{main-theorem}. 


\subsection{Acknowledgments:} We would like to thank R. Bryant and J.M. Landsberg for many discussions over the years on rigidity.

\section{Preliminaries}\label{sec-preliminary}
In this section, we discuss basic facts concerning cominuscule homogeneous varieties and their Schubert classes. We also recall the classification of unobstructed classes given in \cite{robles}.

The classification of cominuscule homogeneous varieties (or, equivalently,  compact complex Hermitian symmetric spaces) is as follows \cite{billey-lakshmibai}, \cite{helgason}:

\begin{enumerate}
\item Grassmannians $G(k,n) = G(k,V)$ parameterizing $k$-dimensional subspaces of an $n$-dimensional vector space $V$,
\item Smooth quadric hypersurfaces in $\PP^{n+1}$,
\item Lagrangian Grassmannians $LG(n,2n) \subset G(n,2n)$ parameterizing subspaces that are isotropic with respect to a non-degenerate skew-symmetric form,
\item Spinor varieties $OG(n,2n) \subset G(n,2n)$ that form an irreducible component of the space of maximal isotropic subspaces with respect to a non-degenerate symmetric form,
\item Two exceptional homogeneous varieties: the Cayley plane $\mathbb{O}\PP^2$ ($E_6/P_6$) and the Freudenthal variety $G(\mathbb{O}^3, \mathbb{O}^6)$ ($E_7/P_7$).
\end{enumerate}
After introducing some notation, we will discuss the Schubert classes in each case.

Given a sequence $\lambda$ of increasing integers $\lambda_1 < \dots < \lambda_s$, we can associate a sequence of non-decreasing integers $\tilde{\lambda}$ by setting $\tilde{\lambda}_i = \lambda_i - i$. For discussing rigidity, it is more convenient to express $\tilde{\lambda}$ as $\mu= (\mu_1^{i_1}, \dots, \mu_t^{i_t})$ by grouping together the terms of $\tilde{\lambda}$ that are equal. More concretely, 
$$\tilde{\lambda}_1 = \cdots = \tilde{\lambda}_{i_1} = \mu_1, \ \tilde{\lambda}_{i_1 + 1} = \cdots = \tilde{\lambda}_{i_1 + i_2} = \mu_2, \ \dots, \ \tilde{\lambda}_{i_1 + \dots + i_{t-1} + 1} = \cdots = \tilde{\lambda}_{s= i_1 + \cdots + i_t } = \mu_t,$$ and $\mu_1 < \mu_2 < \cdots < \mu_t$. We refer to $\mu$ as {\em the sequence associated to $\lambda$}.

Let $Q$ denote a non-degenerate symmetric or skew-symmetric form on an $n$-dimensional vector space $V$. A subspace $W$ of $V$ is called {\em isotropic} with respect to $Q$ if $w^T Q v = 0$ for every $v,w \in W$. Let $W^{\perp}$ denote the subspace $\{v \in V\ | \ w^T Q v =0, \forall w \in W \}$. When $Q$ is a non-degenerate symmetric quadratic form, it defines a smooth quadric hypersurface in $\PP V$. A subspace $W$ is isotropic if and only if $\PP W$ is contained in the quadric hypersurface defined by $Q$. The linear space $\PP W^{\perp}$ is the linear space everywhere tangent to $Q$ along $W$ \cite{coskun:restriction}. 

We denote the cohomology class of a variety $Y$ by $[Y]$.

\subsection{Type A Grassmannians.}  \label{S:Gr}
Schubert classes in $G(k,n)$ are indexed by sequences $\lambda$ of length $k$ such that  $0< \lambda_1 < \lambda_2 < \cdots < \lambda_k \leq n$.  Let $F_{\bullet} = F_1 \subset \cdots \subset F_n$ be a full flag. Then the Schubert variety $\Sigma_{\lambda}(F_{\bullet})$ is defined by $$\Sigma_{\lambda}(F_{\bullet}) := \{ W \in G(k,n) \ | \ \dim (W \cap F_{\lambda_i}) \geq i, \ 1 \leq i \leq k  \}.$$ We denote its cohomology class by $\sigma_{\lambda}$.

The multi rigid Schubert classes for $G(k,n)$ have been classified by Hong \cite{hong:singular} and Robles and The \cite{robles}. Let $\mu= (\mu_1^{i_1}, \dots, \mu_t^{i_t})$ be the sequence associated to $\lambda$. The Schubert class $\sigma_{\lambda}$ is multi rigid (equivalently, unobstructed) if and only if 
\begin{itemize}
\item $i_j \geq 2$ for $1 < j < t$,  
\item $\mu_{j-1} \leq \mu_{j} -2$ for $1< j \leq t$, 
\item $i_1 \geq 2$ if $\mu_1 > 0$, and $i_t \geq 2$ if $\mu_t < n-k$ (\cite{hong:singular}, \cite{robles}, \cite{robles2}, \cite{coskun:orthogonal}). 
\end{itemize}
\medskip

\subsection{Quadric hypersurfaces.} Let $Q$ be a smooth quadric hypersurface in $\PP^{n+1}$. Let $L_j$ denote an isotropic subspace of dimension $j$.  If $n$ is odd, the Schubert varieties in $Q$ are linear subspaces $\PP L_j \subset Q$ for $1 \leq j \leq \frac{n+1}{2}$ and linear sections $\PP L_j^{\perp}\cap Q$ of $Q$ for $0 \leq j < \frac{n+1}{2}$.  If $n$ is even, the space of maximal dimensional isotropic subspaces has two connected components distinguished by the cohomology classes of the linear spaces. Two linear spaces $\PP A, \PP B$ belong to the same connected component if and only if $\dim(A \cap B ) = n\  (\mbox{mod} \  2)$ (see \cite{coskun:restriction}, \cite{joe:thebook} for more details). The Schubert varieties in this case are $\PP L_j$ for $1 \leq j < \frac{n}{2}+1$, $\PP L_{\frac{n}{2}+1}$ and $\PP L_{\frac{n}{2}+1}'$, where $L_{\frac{n}{2}+1}$ and $L_{\frac{n}{2}+1}'$ belong to different connected components of the space of maximal isotropic subspaces, and $\PP L_{j}^{\perp} \cap Q$ for $0 \leq j \leq \frac{n}{2}-1$. 

Hong \cite{hong:rigidity} proves that the point class ($[\PP L_1]$), the fundamental class ($[\PP L_0^{\perp}]$) and, if $n$ is even, the classes of maximal isotropic spaces ($[\PP L_{\frac{n}{2}+1}]$ and $[\PP L_{\frac{n}{2}+1}' ]$) are multi rigid. It is easy to see that all other Schubert classes are flexible (see \cite{coskun:orthogonal} or \cite{robles2}).

\subsection{Lagrangian Grassmannians.} \label{S:LG}
Let $Q$ be a non-degenerate skew-symmetric form on a vector space $V$ of dimension $2n$. Let $LG(n,2n)$ denote the Grassmannian parameterizing Lagrangian subspaces of $V$. Let $0 \leq s \leq n$ be an integer. The Schubert classes in $LG(n,2n)$ are indexed by partitions $\lambda$ such that $$0 < \lambda_1 < \lambda_2 < \cdots < \lambda_s \leq n.$$ Fix an isotropic partial flag $F_1 \subset \cdots \subset F_n $, where $F_i$ is an isotropic subspace of dimension $i$.  The Schubert variety $\Sigma_{\lambda}$ is the Zariski closure of the locus $$\{ W \in LG(n,2n) \ | \ \dim(W \cap F_{\lambda_i}) = i, 1 \leq i \leq s \} .$$ 
Let  $\mu= (\mu_1^{i_1}, \dots, \mu_t^{i_t})$ be the sequence associated to $\lambda$. Robles and The \cite{robles}, \cite{robles2} prove that the Schubert class $\sigma_{\lambda}$ is multi rigid (equivalently, unobstructed) if and only if 
\begin{itemize}
\item $i_j \geq 2$ and $\mu_{j-1} \leq \mu_{j} -2$ for all $1 < j \leq t$, 
\item  $i_1 \geq 2$ if $\mu_1 > 0$, and $\lambda_s \leq n-2$ if $\lambda_s < n$. 
\end{itemize}

\subsection{Spinor varieties.} \label{S:Spin}
Let $Q$ be a non-degenerate symmetric form on a vector space $V$ of dimension $2n$. Let $OG(n,2n)$ denote the spinor variety that forms one of the two isomorphic irreducible components of the space parameterizing $n$-dimensional isotropic subspaces of $V$. 
Schubert varieties are indexed by increasing sequences $1 \leq \lambda_1 < \cdots < \lambda_s \leq n-1$ of length at most $n-1$. Let $F_1 \subset \cdots \subset F_{n-1}$ be an isotropic  partial  flag. The Schubert variety is the Zariski closure of the locus $$\{ W \in OG(n,2n) \ | \ \dim(W \cap F_{\lambda_i}) = i, 1 \leq i \leq s \}.$$ 
Let $\mu= (\mu_1^{i_1}, \dots, \mu_t^{i_t})$ be the sequence associated to $\lambda$. Robles and The \cite{robles}, \cite{robles2} prove that $\sigma_{\lambda}$ is multi rigid (equivalently, unobstructed) if and only if 
\begin{itemize}
\item $i_j \geq 2$ and $\mu_{j-1} \leq \mu_{j} -2$ for $1 < j \leq t$, 
\item $i_1\geq 2$ if $\mu_1 >0$ and  $\lambda_s \leq n-3$ if $\lambda_s < n-1$.
\end{itemize}

\begin{remark} \label{R:LGS}
Observe that Schubert classes in $LG(n,2n)$ and $OG(n+1, 2n+2)$ are indexed by the same set of increasing sequences. The sequences that index multi rigid (equivalently, unobstructed) classes in the two varieties coincide. 
\end{remark}

\subsection{The Cayley plane and the Freudenthal variety.} \label{S:excep}
The Cayley plane $\mathbb{O}\PP^2 = E_6/P_6$ is the closed $E_6$--orbit in $\mathbb{P}V$, where $V$ is an irreducible representation of $E_6$ of dimension 27.
It is a 16-dimensional cominuscule homogeneous variety. Similarly, the Freudenthal variety $G(\OOO^3, \OOO^6) = E_7/P_7$ is the closed orbit in $\mathbb{P}V$, where $V$ is the irreducible $E_7$--representation of dimension 56. It is a 27-dimensional cominuscule homogeneous variety. We refer the reader to \cite{iliev}, \cite{jm1}, \cite{jm2}, \cite{jm3} and \cite{nikolenko} for more details on the discussion in this section. 

The Schubert classes in $\OOO \PP^2$ and $G(\OOO^3, \OOO^6)$ are indexed by the nodes of their Hasse diagrams (see Figures \ref{f:E6} and \ref{f:E7} on pages \pageref{f:E6} and \pageref{f:E7}). In the figures, the classes of the Schubert varieties of the same dimension lie at the same height. The dimension of the variety decreases as one proceeds from top to bottom. Closure relations among the Schubert cells is indicated by the edges in the diagram. A Schubert cell of dimension $d$ is contained in the closure of a Schubert cell of dimension $d+1$ if and only if there is an edge in the diagram connecting the corresponding nodes. The nodes that have been encircled correspond to the multi rigid Schubert classes \cite{robles}, \cite{robles2}. The numbers to the right of each node denote the degrees of the Schubert varieties in the minimal embedding of $\OOO \PP^2$ and $G(\OOO^3, \OOO^6)$ in $\PP^{26}$ and $\PP^{55}$, respectively. We will explain the additional decorations in the diagrams as the need arises. 

We will refer to the Schubert classes by a pair of integers $(\dim: \deg)$, where the first one denotes the dimension of the variety and the second one denotes the degree of the variety in the minimal embedding.  (The cominuscule Cayley plane and Freudenthal variety are also minuscule, so the degrees of their Schubert varieties may be computed using \cite[\S4.8.4]{billey-lakshmibai}.)  This notation is unambiguous except for the two linear spaces $\PP^4$ in the Cayley plane and the two linear spaces $\PP^5$ in the Freudenthal variety. We will distinguish between these two classes by specifying whether they are maximal linear spaces.  

\subsection{Tits correspondences} \label{S:tits}

Tits correspondences clarify the geometry of a rational homogeneous variety $G/P$.  These correspondences will be especially useful in our discussion of the exceptional Cayley plane and Freudenthal variety.  Fix a Borel subgroup $B \subset G$.  Every parabolic subgroup of $G$ is conjugate to a subgroup containing $B$.  The subgroups containing $B$ are the \emph{standard parabolics}.  The standard parabolics are in one-to-one correspondence with the subsets of the nodes in the Dynkin diagram of $G$.  Our convention is that a single node corresponds to a maximal parabolic, and the full Dynkin diagram to the Borel.  Given a subset $I\subset \{1,\ldots,\mathrm{rank}(G)\}$ indexing the nodes, let $P = P_I$ be the corresponding parabolic subgroup.  So, for example, the Cayley plane $\cayley$ is $E_6/P_6$, where $P_6$ is the maximal (standard) parabolic associated to the sixth node.  (The Cayley plane is isomorphic to the \emph{dual Cayley plane} $E_6/P_1$.)  Similarly, the Freudenthal variety $\freud$ is $E_7/P_7$, where $P_7$ is the maximal parabolic associated to the seventh node of the Dynkin diagram.  Given two standard parabolics $P=P_I$ and $Q=P_J$, the intersection $P \cap Q = P_{I\cup J}$ is also a standard parabolic.  There is a natural double fibration, called the \emph{Tits correspondence}, given by the diagram: 
\begin{center} \setlength{\unitlength}{4pt}
\begin{picture}(33,11)
\put(0,0){$G/P$} 
\put(10,6.5){\vector(-1,-1){4}}
\put(10,8.5){$G/(P\cap Q)$}
\put(21.5,6.5){\vector(1,-1){4}}
\put(25,0){$G/Q$,}
\put(6,5.5){$\eta$}
\put(23.5,5){$\tau$}
\end{picture}
\end{center}
where the maps $\eta$ and $\tau$ are the natural projections. Given a subset $\Sigma \subset G/Q$, {\em the Tits transform} is $\mathcal{T}(\Sigma) : = \eta(\tau^{-1}(\Sigma))$.  

The Tits transform of a point $o \in G/Q$ will play a crucial r\^ole in our discussion of Schubert varieties in the Cayley plane and Freudenthal variety.  These transforms have been studied by Landsberg and Manivel \cite{jm1}, and we review the relevant results in Examples \ref{eg:E6a} and \ref{eg:E7a}.

\begin{example} \label{eg:E6a}
Let $G/P = E_6/P_6 = \cayley$ be the Cayley plane.
\begin{a_list}
\item The Tits transformation of a point $o \in E_6/P_5$ is a line $\bP^1 \subset E_6/P_6$ in the Cayley plane.  It follows that $\cayley$ is uniruled by lines, and the uniruling is parameterized by $G/Q=E_6/P_5$.

\item Similarly, the Tits transformation of a point $o \in E_6/P_1$ in the dual Cayley plane is a smooth 8-dimensional quadric $\cQ^8 \subset E_6/P_6$.  In particular, the $\cayley$ is uniruled by quadrics, and the uniruling is parameterized by $G/Q=E_6/P_1$.
\end{a_list}
\end{example}

\begin{example} \label{eg:E7a}
Let $G/P = E_7/P_7 = \freud$ be the Freudenthal variety.
\begin{a_list}
\item The Tits transformation of a point $o \in E_7/P_6$ is a line $\bP^1 \subset E_7/P_7$ in the Freudenthal variety.  It follows that $\freud$ is uniruled by lines, and the uniruling is parameterized by $G/Q=E_7/P_6$.

\item Similarly, the Tits transformation of a point $o \in E_7/P_1$ is a smooth 10-dimensional quadric $\cQ^{10} \subset E_7/P_7$.  In particular, $\freud$ is uniruled by quadrics, and the uniruling is parameterized by $G/Q=E_7/P_1$.
\end{a_list}
\end{example}

More generally, we may use the Weyl group and (generalized) Bruhat decomposition of $G$ to compute the Tits transform of a Schubert variety (Lemma \ref{L:tits} below).  What follows is standard and there are many good references, for example \cite{springer}.  A parabolic subalgebra $\fp$ is the direct sum $\fp_0 \op \fp_+$ of a reductive subalgebra $\fp_0$ and a nilpotent subalgebra $\fp_+$.  Let $W$ denote the Weyl group of $\fg$, and identify the Weyl group of $\fp_0$ with a subgroup $W_\fp \subset W$.  Given two elements $v,w \in W$, the subsets $B v^{-1} P$ and $B w^{-1}P$ are either disjoint or equal.  Moreover, they are equal if and only if $W_\fp v = W_\fp w$.  The (generalized) \emph{Bruhat decomposition} of $G$ associated to the pair $B \subset P$ is the disjoint union 
$$
  G \ = \ \bigsqcup_{W_\fp w \in W_\fp\backslash W} B w^{-1} P \,.
$$

Given $w \in W$, each $C_w = B w^{-1} P/P \subset G/P$ is a \emph{Schubert cell}.  As above, $C_w = C_v$ if and only if $W_\fp v = W_\fp w$, and when equality fails the cells are disjoint.  In particular, the homogeneous variety $G/P$ is a disjoint union 
$$
  G/P \ = \ \bigsqcup_{W_\fp w \in W_\fp\backslash W} C_w \,.
$$
of Schubert cells indexed by $W_\fp\backslash W$.  The \emph{Schubert varieties} of $G/P$ are the Zariski closures $X_w = \overline{C}{}_w$.

Each right coset $W_\fp\backslash W$ admits a unique representative of minimal length.  Let $W^\fp_\mathrm{min} \subset W$ denote this set of representatives.  Similarly each right coset admits a unique representative of maximal length.  Let $W^\fp_\mathrm{max} \subset W$ denote this set of representatives.  If $w_0$ is the longest element of $W_\fp$, then $W^\fp_\tmax = \{ w_0 w \ | \ w \in W^\fp_\tmin\}$.  We have $W_\fp \backslash W \simeq W^\fp_\tmin \simeq W^\fp_\tmax$.  If $w \in W^\fp_\tmin$, then the dimension of $X_w$ is the length of $w$.

Let $R = P\cap Q$ and $v,w \in W$.  Then $B v^{-1} R \subset B w^{-1} P$ if and only if $W_\fp v = W_\fp w$.  (If containment fails, then the sets are disjoint.)  Similarly, $B v^{-1} R \subset B w^{-1} Q$ if and only if $W_\fq v = W_\fq w$.  This has the following implications for the images of Schubert varieties under $\tau^{-1}$ and $\eta$.  First, given $w \in W^\fq_\tmax$, let $Y_w \subset G/Q$ denote the corresponding Schubert variety.  Then $\tau^{-1}(Y_w)$ is the Schubert variety $Z_{w} \subset G/R$ associated to $W_\fr w$.  (Here, it is essential that $w$ is a maximal length representative of $W_\fq w \in W_\fq\backslash W$.  If $w'$ is a second representative (e.g. $w' \in W^\fq_\tmin$), then $Z_{w'} \subset Z_w$ and, in general, containment is strict.)  Second, $\eta(Z_w) = \cT(Y_w)$ is the Schubert variety $X_w$ associated to $W_\fp w$.  We summarize this discussion with the following

\begin{lemma} \label{L:tits}
Let $w \in W^\fq_\tmax$, and let $Y_w \subset G/Q$ denote the Schubert variety indexed by the coset $W_\fq w$.  Then the Tits transformation $\cT(Y_w)$ is the Schubert variety $X_w \subset G/P$ indexed by the coset $W_\fp w$.
\end{lemma}

Note that the point $o \in G/Q$ of Examples \ref{eg:E6a} and \ref{eg:E7a} is a Schubert variety.  So the Tits transformations $\bP^1 , \cQ^8 \subset \cayley$ and $\bP^1 , \cQ^{10} \subset \freud$ of those examples are Schubert varieties.  More generally, Lemma \ref{L:tits} yields descriptions of many Schubert varieties in the Cayley plane and Freudenthal variety as uniruled by lines/quadrics, and with the unirulings parameterized by Schubert varieties in $G/Q$; see Examples \ref{eg:E6b} and \ref{eg:E7b}.

\begin{example} \label{eg:E6b}
Let $G/P = E_6/P_6 = \cayley$ be the Cayley plane.
\begin{a_list}
\item In Example \ref{eg:E6a}.a we noted that the Cayley plane is uniruled by lines.  The lines on $\OOO \PP^2$ passing through a point are parameterized by the Spinor variety $Y_w = OG(5,10) \subset E_6/P_5$. These lines sweep out the Schubert variety $\cT(Y_w) = (11:12) \subset \cayley$, which is isomorphic to the cone $C(OG(5,10))$ over $OG(5,10)$, \cite{iliev}; see Figure \ref{f:E6}.  

Likewise, the Schubert variety $(12:45) \subset \cayley$ is the locus swept out by lines in $\OOO \PP^2$ intersecting a fixed line in $\OOO \PP^2$.  Indeed, $(12:45) = \cT(Y_v)$, where $OG(5,10)$ is a Schubert divisor of $Y_v \subset E_6/P_5$.

\item In Example \ref{eg:E6a}.b we noted that the Cayley plane is uniruled by quadrics.  There exist Schubert varieties $\bP^1 , \bP^2 \subset E_6/P_1$ in the dual Cayley plane such that $\cT(\bP^1) = (9:9)$ and $\cT(\bP^2) = (10:21)$.  That is, the Schubert varieties $(9:9)$ and $(10:21)$ in $\cayley$ are swept out by a pencil and a net of quadrics, respectively.
\end{a_list}
\end{example}

\begin{example} \label{eg:E7b}
Let $G/P = E_7/P_7 = \freud$ be the Freudenthal variety.
\begin{a_list}
\item In Example \ref{eg:E7a}.a we noted that $\freud$ is uniruled by lines.   Lines passing through a point in $G(\OOO^3,\OOO^6)$ are parameterized by a Cayley plane $Y_w = \cayley \subset E_7/P_6$. These lines sweep out the Schubert variety $\cT(Y_w) = (17:78)$, which is isomorphic to the cone over $\OOO \PP^2$ and corresponds to the node in Figure \ref{f:E7} marked by $C(\OOO \PP^2)$.

The Schubert variety $(18:520) \subset \freud$ is the Tits transform $\cT(Y_v)$ of a Schubert variety $Y_v \subset E_7/P_6$ containing the Cayley plane as a divisor, and is swept out by lines on $G(\OOO^3, \OOO^6)$ that intersect a fixed $\bP^1 \subset G(\OOO^3, \OOO^6)$.

\item In Example \ref{eg:E7a}.b we noted that the $\freud$ is uniruled by quadrics.  Table \ref{t:tits} lists the Schubert varieties in $G(\OOO^3, \OOO^6)$ swept out by quadrics parameterized by Schubert varieties in $\OOO \PP^2 \subset E_7/P_1$.  The multi rigid classes are indicated by ${}^*$.  
\begin{table}[h] \renewcommand{\arraystretch}{1.2}
\caption{Tits transformations of $Y \subset \mathbb{O}\PP^2\subset E_7/P_1$ to $\mathcal{T}(Y_v) \subset G(\mathbb{O}^3,\mathbb{O}^6)$.}
\label{t:tits}
\begin{tabular}{|c|c||c|c||c|c|}
  \hline
  {$Y$} & {$\mathcal{T}(Y)$} &
  {$Y$} & {$\mathcal{T}(Y)$} &
  {$Y$} & {$\mathcal{T}(Y)$} \\ \hline \hline
  $0:1^*$ & $10:2^*$ & $1:1$ & $11:11$ & $2:1$ & $12:32$ \\ \hline
  $3:1$ & $13:65$  & $4:1^*$ & $14:110^*$ & $4:1$ & $14:98$ \\ \hline
  $5:2$ & $15:286$ & $5:1^*$ & $15:98^*$ & $6:2$ & $16:364$ \\ \hline
  $6:3$ & $16:384$ & $7:2$ & $17:442$ & $7:5$ & $17:748$ \\ \hline
  $8:2^*$ & $18:520$ & $8:7$ & $18:1190$ & $8:5^*$ & $18:748^*$ \\ \hline
  $9:9$ & $19:1710$ & $9:12$ & $19:1938$ & $10:21$ & $20:3648$ \\ \hline
  $10:12$ & $20:1938$ & $11:33$ & $21:5586$ & $11:12^*$ & $21:1938^*$ \\ \hline 
  $12:45$ & $22:7524$ & $12:33^*$ & $22:5586^*$ & $13:78$ & $23:13110$ \\ \hline
  $14:78$ & $24:13110$ & $15:78$ & $25:13110$ & $16:78^*$ & $26:13110$ \\ \hline
\end{tabular}
\end{table}
\end{a_list}
\end{example}

\noindent We will use the concrete interpretations of Schubert varieties in $\OOO \PP^2$ and $G(\OOO^3, \OOO^6)$ to define irreducible subvarieties that represent multiples of the flexible Schubert classes.

\section{Constructions of irreducible subvarieties}\label{sec-proof}

In this section, we discuss three methods for constructing irreducible subvarieties of homogeneous varieties. These constructions suffice to prove Theorem \ref{main-theorem}.

\subsection{Construction 1: Bertini's Theorem.} Bertini's theorem asserts that a general member of a very ample linear system on a variety of dimension at least two is irreducible \cite[III.10.9]{hartshorne}. Let $X$ be a homogeneous variety embedded in projective space by its minimal homogeneous embedding. Let $Y$ be an irreducible subvariety of $X$ of dimension at least two. Let $Z_d$ be the intersection of $Y$ with a general hypersurface of degree $d$. By Bertini's theorem, $Z_d$ is irreducible. Since the class of a hypersurface of degree $d$ is $d$ times the class of a hyperplane, we conclude that $Z_d$ is an irreducible representative of the class $d [Z_1]$. Typically, the class $[Z_1]$ is not a multiple of  a Schubert class. However, if $Y$ is a Schubert variety that has only one Schubert divisor class $\sigma$, then necessarily $[Z_1]$ is a multiple of $\sigma$. When the multiple is one, we conclude that we can represent every positive multiple of $\sigma$ by an irreducible subvariety. We thus conclude the following proposition.

\begin{proposition}\label{prop-Bertini}
Let $Y \subset X$ be a Schubert variety of dimension at least two that has only one Schubert divisor class $\sigma$. Suppose that a general hyperplane section of $Y$ has class $m \sigma$. Then for every integer $d \geq 1$, the class $dm \sigma$ can be represented by an irreducible subvariety.
\end{proposition}

Applying Proposition \ref{prop-Bertini} when $Y$ is a linear space, we obtain the following corollary.

\begin{corollary}\label{cor-linear}
Let $\sigma$ be the class of a non-maximal linear space of dimension at least one in a homogeneous variety $X$. Then $m \sigma$ can be represented by an irreducible subvariety for every $m \geq 1$.
\end{corollary}

When $X$ has Picard rank one, applying Proposition \ref{prop-Bertini} to $X$, we obtain the following corollary.

\begin{corollary}\label{cor-divisor}
Let $X$ be a homogeneous variety of dimension at least two and Picard number one. Let $\sigma$ be the class of a Schubert divisor. Then $m \sigma$ can be represented by an irreducible subvariety for every $m \geq 1$.
\end{corollary}
 
Let $Q \subset \PP^{n+1}$ be a smooth quadric hypersurface. Since $H^{2k}(Q, \ZZ) = \ZZ$ for $0 \leq k \leq n$ and $k \not= \frac{n}{2}$, Proposition \ref{prop-Bertini} is well-suited for proving the flexibility of Schubert classes in the cohomology of $Q$.
 
\begin{corollary}\label{cor-quadric}
Let $\sigma$ be a Schubert class on a smooth quadric hypersurface $Q \subset \PP^{n+1}$. Assume that $\sigma$ is not the fundamental class, the class of a point or the class of a maximal dimensional linear space when $n$ is even. Then $m \sigma$ can be represented by an irreducible subvariety of $Q$ for every $m \geq 1$. 
\end{corollary} 

\begin{proof}
The Schubert classes $\sigma = [\PP L_j]$ for $2 \leq j \leq \frac{n}{2}$ are the classes of positive-dimensional non-maximal linear spaces on $Q$. By Corollary \ref{cor-linear}, $m \sigma$ can be represented by an irreducible subvariety of $Q$. More concretely, $m \sigma$ is represented by an irreducible hypersurface of degree $m$ in $\PP L_{j+1}$.

The Schubert class $\sigma= [\PP L_j^{\perp} \cap Q]$ for 
 $1 \leq j < \frac{n}{2}$ 
is the unique Schubert divisor class in the Schubert variety $\PP L_{j-1}^{\perp} \cap Q$. Furthermore, $\PP L_{j}^{\perp} \cap Q$ is the intersection of $\PP L_{j-1}^{\perp} \cap Q$ with the tangent hyperplane at a point $x \in \PP L_{j} - \PP L_{j-1}$.  By Proposition \ref{prop-Bertini}, we conclude that $m \sigma$ can be represented by an irreducible subvariety. More concretely, $m \sigma$ is represented by the intersection of $Q$ with a general hypersurface of degree $m$ contained in $\PP (L_{j-1}^{\perp})$ (or, more generally, in a linear space of codimension $j-1$).

It remains to consider the case that $n$ is odd and $\sigma = [\PP L_{\frac{n+1}{2}}]$ is the class of a maximal linear space.  When $n$ is odd, a linear section of the Schubert variety $\PP L_{\frac{n-1}{2}}^{\perp} \cap Q$ has class $2 [\PP L_{\frac{n+1}{2}}]$. Hence, when $\sigma$ is the class of a maximal linear space, Proposition \ref{prop-Bertini} only implies that $2 m \sigma$ can be represented by an irreducible subvariety of $Q$ for every $m \geq 1$.  Since $H^{n+1}(Q, \ZZ)= \ZZ$, to see that all multiples $m \sigma$ can be represented by an irreducible subvariety, it suffices to construct an irreducible variety of dimension $\frac{n-1}{2}$ and degree $m$ in $Q$. First, a smooth quadric surface in $\PP^3$ is isomorphic to $\PP^1 \times \PP^1$ and contains irreducible curves of any positive degree. To obtain an irreducible curve of degree $m$, take the graph of a morphism from $\PP^1$ to $\PP^1$ of degree $m-1$, or equivalently, the zero locus of an irreducible bihomogeneous polynomial of bidegree $(1,m-1)$. By taking a general hyperplane section, we conclude that a smooth quadric threefold contains irreducible curves of any positive degree $m$. By induction, assume that a quadric $Q$ in $\PP^{n+1}$ contains an irreducible subvariety $Y$ of dimension $\frac{n-1}{2}$ and degree $m$.  Let $Q' \subset \PP^{n+3}$ be a smooth quadric hypersurface. The tangent hyperplane section $T_p \cap Q'$ of $Q'$ at a point $p$ is a cone over a smooth quadric $Q$ in $\PP^{n+1}$. The cone over $Y \subset Q$ with vertex at $p$ is then an irreducible subvariety of $Q'$ of dimension $\frac{n+1}{2}$ and degree $m$. This concludes the induction and the proof of the corollary.
\end{proof}

\begin{remark}
Let $\sigma$ be a Schubert class in a smooth quadric hypersurface $Q \subset \PP^{n+1}$ other than the fundamental class, the class of a point or the class of a maximal dimensional linear space.  The proof of Corollary \ref{cor-quadric} shows that $m \sigma$ can be represented by a smooth subvariety of $Q$. 
\end{remark}

Proposition \ref{prop-Bertini} also applies to many Schubert classes in the exceptional cominuscule varieties. The reader should refer to Figures \ref{f:E6} and \ref{f:E7} for the next corollary.

\begin{corollary}\label{cor-Bertini-exceptional}
Let $\sigma$ be a Schubert class marked by a $+$ in the Hasse diagram of the Cayley plane or the Freudenthal variety (see Figures \ref{f:E6} and \ref{f:E7}). Then $m \sigma$ can be represented by an irreducible subvariety for every $m \geq 1$.
\end{corollary}

\begin{proof}
Let $\sigma$ be a class marked with a $+$ in the Hasse diagram of $\mathbb{O}\PP^2$ or $G(\mathbb{O}^3, \mathbb{O}^6)$. Let $\Sigma$ be a Schubert variety with class $\sigma$.
A quick inspection of the Hasse diagrams shows that there exists a Schubert variety $\Sigma'$ with class $\sigma'$ such that $\sigma$ is the unique Schubert divisor class in $\Sigma'$, $\deg(\Sigma) = \deg(\Sigma')$ and $\dim(\Sigma') = \dim(\Sigma) + 1$. Consequently, the class of a hyperplane section of $\Sigma'$ is equal to $\sigma$ and the corollary follows by Proposition \ref{prop-Bertini}.
\end{proof}

\begin{remark}
For applications, it is useful to observe that the subvarieties $Y$ constructed by Proposition \ref{prop-Bertini} are contained in a Schubert variety $\Sigma'$ of one larger dimension. Moreover, $Y$ can be taken to contain a general point of $\Sigma'$. Furthermore, if $\Sigma'$ is smooth, we can take $Y$ to be smooth. 

For completeness, we remark that if $\sigma$ is the class of a maximal linear space on $Q\subset \PP^{n+1}$ with $n$ odd, we can find irreducible representatives of $m \sigma$ contained in the Schubert variety $\PP L_{\frac{n-1}{2}}^{\perp} \cap Q$. To see this, we need to modify the proof of Corollary \ref{cor-quadric} by constructing irreducible curves on an irreducible quadric cone $C$ in $\PP^3$. The cone $C$ is the image of the surface $F_2 = \PP(\OO_{\PP^1} \oplus \OO_{\PP^1}(-2))$ under the linear system $|\OO_{F_2}(1)|$. The curve classes on $F_2$ are generated by the class of a fiber $f$ and the class of the exceptional curve $e$. The linear systems $|e + mf|$ contain irreducible curves for $m =0$ or $m \geq 2$. The images of these curves have degree $m$ in $\PP^3$ (see \cite{coskun:scrolls} for more details). Taking these curves as the base of our induction, we can ensure that the irreducible varieties representing $m \sigma$ constructed  in the proof of Corollary \ref{cor-quadric} are contained in $\PP L_{\frac{n-1}{2}}^{\perp} \cap Q$.
\end{remark}

\subsection{Construction 2: Moduli spaces.} Homogeneous varieties such as $G(k,n)$, $LG(n,2n)$ or $OG(n,2n)$ represent functors. We can define subvarieties of these homogeneous varieties by defining geometric subfunctors. Typical examples of such subfunctors include linear spaces satisfying incidence and tangency conditions with respect to a subvariety of projective space. The class of such a subvariety will  be the multiple of a Schubert class only in very special circumstances. We now explain the fundamental example in $G(k,n)$. Special cases of this  construction have already appeared in \cite{coskun:rigid}, \cite{coskun:orthogonal} and \cite{robles2}.

Let $\sigma_{\lambda}$ be a Schubert class in $G(k,n)$. Assume that the associated sequence $\mu= (\mu_1^{i_1}, \dots, \mu_t^{i_t})$ has an index $j$ such that  $i_j =1$ for some $1 \leq j \leq t$ and $n-k > \mu_j > 0$. Set $u = i_1 + i_2 + \cdots + i_j$.
The conditions $i_j = 1$ and $n-k > \mu_j > 0$ are equivalent to $\lambda_{u-1} + 1 < \lambda_u < \lambda_{u+1} - 1$.  (Here, if $u = i_1 = 1$, then set $\lambda_0 = 0$; and if $u = i_1+\cdots+i_t = k$, then set $\lambda_{k+1} = n+1$.)  
Let $$F_{\lambda_1} \subset \cdots \subset F_{\lambda_{u -1}} \subset F_{\lambda_u +1} \subset F_{\lambda_{u+1}} \subset \cdots \subset F_{\lambda_k}$$ be a partial flag, where the dimension of a flag element $F_{d}$ is $d$. The dimensions of the flag elements in this partial flag are the same as those that define the Schubert class $\sigma_{\lambda}$, except at the $u$-th flag where it is one greater. Fix a smooth plane curve $C$ of degree $m$ in $\PP F_{\lambda_u +1}$ whose span is disjoint from $\PP F_{\lambda_u-2}$. Since the difference in dimensions between $\PP F_{\lambda_u +1}$ and $\PP F_{\lambda_u-2}$ is $3$, we can find such a plane curve $C$.  The assumption $i_j =1$ implies that $F_{\lambda_{u-1}} \subseteq F_{\lambda_{u} -2}$ and $F_{\lambda_u +1} \subsetneq F_{\lambda_{u+1}}$.  Let $Y$ be the cone over $C$ with vertex $\PP F_{\lambda_u -2}$.  (If $u = i_1 = 1$, then $1 < \lambda_1$.  If $\lambda_1 = 2$, then $Y = C$.)

Define $Z$ to be the Zariski closure of the locus $Z^0$ in $G(k,n)$ parameterizing $k$-planes $W$ such that $\PP W \cap Y$ is a projective linear space of dimension $u-1$ and  $\dim(W \cap F_{\lambda_j}) = j$ for $j \not= u$. Then it is straightforward to verify that $Z$ is irreducible and has cohomology class $m \sigma_{\lambda}$. First, the locus $U$ parameterizing linear spaces of dimension $u$ contained in $Y$ and satisfying $\dim (W \cap F_{\lambda_j}) = j$ for $j < u$ is irreducible. This variety maps onto a Zariski dense open subset of  the Schubert variety $\Sigma_{\lambda_1, \dots, \lambda_{u-1}}$ in  $G(u-1, F_{\lambda_{u-1}})$ by taking the intersection of the linear spaces with $F_{\lambda_{u-1}}$. The fibers correspond to a choice of point on the cone $Y$ away from $\PP F_{\lambda_{u-1}}$.  Hence, the locus $U$ fibers over an irreducible variety with equidimensional irreducible fibers. By the theorem on the dimension of fibers \cite[I.8.6]{shafarevich},  we conclude that $U$ is irreducible. Next, the locus $Z^0$ maps onto $U$ by intersecting a linear space parameterized by $Z^0$ with $Y$. The fibers over a subspace $W' \in U$ are open dense subsets of Schubert varieties in $G(k-u, V/W')$. Hence, by the theorem on the dimension of fibers, we conclude that $Z$ is irreducible. 

To compute its class, we can intersect $Z$ with complementary dimensional Schubert varieties. Recall that the Poincar\'{e} dual Schubert class is given by the sequence $\lambda^*$, where $\lambda_i^* = n - \lambda_{k-i+1} + 1$. If $\sigma_{\nu}$ is the class of a complementary dimensional Schubert variety and $\nu \not= \lambda^*$, then we claim that $[Z] \cap \sigma_{\nu} = 0$.  To see this, not that since the Schubert varieties indexed by $\nu$ and $\lambda^*$ have the same dimension, there exists an index $i$ such that $\nu_i < \lambda_i^*$. If we take a general representative $\Sigma_{\nu}$ of $\sigma_{\nu}$ defined by a flag $G_{\bullet}$, since $\PP G_{\nu_i}$ does not intersect  $\PP F_{\lambda_{k-i+1}}$ (or $Y$ if $i=k-u+1$), the intersection $Z \cap \Sigma_{\nu}$ is empty. 

On the other hand, $[Z] \cap \sigma_{\lambda^*} = m$. Take a general Schubert variety $\Sigma_{\lambda^*}$ with class $\sigma_{\lambda^*}$ defined with respect to a flag $G_{\bullet}$. For $i\not= u$, $F_{\lambda_i}$ and $G_{\lambda_{k-i+1}^*}$ intersect in a one-dimensional subspace, which must be contained in any linear space $W \in Z \cap \Sigma_{\lambda^*}$. The variety $Y$ and $\PP G_{\lambda_{k-u+1}^*}$ intersect in $m$ reduced points. 
A subspace $W \in Z \cap \Sigma_{\lambda^*}$ must contain  a one-dimensional subspace of $G_{\lambda_{k-u+1}^*}$ corresponding to one of these $m$ points.  
Conversely, the span of the one-dimensional subspaces  $F_{\lambda_i} \cap G_{\lambda_{k-i+1}^*}$, $i \not= u$, and one of the one-dimensional subspaces corresponding to one of the points in $Y \cap \PP G_{\lambda_{k-u+1}^*}$ is a linear space $W \in Z \cap \Sigma_{\lambda^*}$. By Kleiman's Transversality Theorem \cite{kleiman:transverse}  in characteristic zero  or a simple tangent space calculation in general, we conclude that $Z \cap \Sigma_{\lambda^*}$ consists of $m$ reduced points.  We thus obtain the following proposition.

\begin{proposition}\label{prop-Grassmannian}
Let $\sigma$ be an obstructed Schubert class in the cohomology of $G(k,n)$ and let $m$ be a positive integer. Then there exists an irreducible subvariety of $G(k,n)$ that represents $m \sigma$. 
\end{proposition}

\begin{proof}
By Section \ref{S:Gr}, the obstructed Schubert classes are those classes $\sigma_\lambda$ where the associated sequence $(\mu_1^{i_1}, \dots, \mu_t^{i_t})$ has an index $j$ such that 
\begin{enumerate}
\item either $i_j =1$ and $n-k > \mu_j > 0$,
\item or $\mu_{j+1} - \mu_j =1$.
\end{enumerate}  
In the first case, our construction gives an irreducible subvariety of $G(k,n)$ representing $m \sigma$. In the second case, one can either give an analogous explicit construction or reduce to the first case as follows. There is an isomorphism $\phi:  G(k,n) \rightarrow G(n-k, n)$ that maps a subspace $W \in G(k,n) = G(k,V)$ to $\mathrm{Ann}(W) \in G(n-k,n)= G(n-k,V^*)$ to the linear forms vanishing on $W$.  The isomorphism takes a Schubert class $\sigma$ satisfying condition (2) to a Schubert class $\mu$ satisfying condition (1), cf.  \cite{bryant:rigidity}. Let $Z$ be an irreducible subvariety of $G(n-k,n)$ representing $m \mu$, then $\phi^{-1}(Z)$ is an irreducible subvariety of $G(k,n)$ that represents $m \sigma$. This concludes the proof of Theorem \ref{main-theorem} in the case of Grassmannians.
\end{proof}

\begin{remark}
Given an obstructed class $\sigma$, Proposition \ref{prop-Grassmannian} constructs an irreducible representative $Z$ of $m \sigma$ in a Schubert variety of dimension one larger. More concretely, the variety constructed is contained in the Schubert variety $\Sigma_{\lambda'}$ where $\lambda_i' = \lambda_i$ if $i \not= u$ and $\lambda_u' = \lambda_u + 1$.  Furthermore, $Y$ can be taken to pass through a general point of $\Sigma_{\lambda'}$.
\end{remark}

We can apply the same construction to $LG(n,2n)$ and $OG(n,2n)$ to obtain the following corollary.

\begin{corollary}\label{cor-isotropic}
Let $\sigma$ be an obstructed class in the cohomology of $LG(n,2n)$ or $OG(n,2n)$. Then $m \sigma$ can be represented by an irreducible subvariety for every $m \geq 1$.
\end{corollary}

\begin{proof}
Let $X$ be either $LG(n,2n)$ or $OG(n+1,2n+2)$.  Let $\lambda_1 < \cdots < \lambda_s \leq n$ be a sequence indexing an obstructed Schubert class $\sigma$.  (See Sections \ref{S:LG} and \ref{S:Spin} and Remark \ref{R:LGS}.)  Let $(\mu_1^{i_1}, \dots, \mu_t^{i_t})$ be the associated sequence. First, suppose there exists an index $1 < j \leq t $ such that $\mu_j = \mu_{j-1} +1$ or an index $1 \leq j < t$ such that $i_j =1$ and $\mu_1 > 0$ if $j=1$. If $X=OG(n+1, 2n+2)$, we also allow the case $i_t =1$. If $X=LG(n,2n)$, let $F \in X$ be a maximal isotropic linear space. 
If $X=OG(n+1, 2n+2)$, and $s \equiv n+1$ (mod $2$), then let $F\in OG(n+1,2n+2)$ be a maximal isotropic space; otherwise, if $s \not\equiv n+1$ (mod $2$), let $F$ be a maximal isotropic space of $\mathbb{C}^{2n+2}$ corresponding to a point in the second component of the Isotropic Grassmannian (cf. Section \ref{S:Spin}).  Then $\sigma_\lambda$ is an obstructed class in $G(s,F)$; see Section \ref{S:Gr}.
By Proposition \ref{prop-Grassmannian}, there exists an irreducible subvariety $Z$ of $G(s, F)$ representing $m \sigma_{\lambda_1, \dots, \lambda_s}$. Let $Y$ be the Zariski closure of the locus in $X$ parameterizing maximal isotropic spaces that intersect $F$ in an $s$-dimensional subspace parameterized by $Z$. By the theorem on the dimension of fibers, $Y$ is irreducible. By pairing with Schubert classes in complementary dimension, we see that $[Y] = m \sigma$. 

Next, suppose that $X= OG(n+1, 2n+2)$ and $\lambda_s = n-1$.  If $s \not\equiv n+1$ (mod $2$), let $F \in OG(n+1,2n+2)$ be a maximal isotropic subspace contained in $OG(n+1, 2n+2)$; if $s \equiv n+1 \ (\mbox{mod} \ 2)$, let $F$ be a maximal isotropic subspace in the second component of the Isotropic Grassmannian. 
Then the Schubert class $\sigma_{\lambda'}$ in $G(s+1, F)$, where  $\lambda'$ is the sequence obtained by adding $n+1$ to $\lambda$, is an obstructed class. By Proposition \ref{prop-Grassmannian}, there exists a subvariety $Z$ in $G(s+1, F)$ with class $m \sigma_{\lambda'}$. As in the previous paragraph, define $Y$ to be the Zariski closure of the locus in $X$ parameterizing maximal isotropic spaces that intersect $F$ in an $s+1$-dimensional subspace parameterized by $Z$. Then $Y$ is an irreducible subvariety that represents $m\sigma_{\lambda}$ in $OG(n+1, 2n+2)$.

It remains to address the following two cases when $X=LG(n,2n)$: 
\begin{enumerate}
\item  $\mu_t= n -s$ and $i_t =1$,
\item $\lambda_s = n-1$ (equivalently, $\mu_t = n-1-s$).
\end{enumerate} 
In both cases, let $H$ be the cone with vertex $\PP F_{n-2}$ over a general irreducible plane curve of degree $m$ in $F_{n-2}^{\perp}$.  Let  $Y$ be the Zariski closure of the locus parameterizing isotropic subspaces $W$ such that $\PP W \cap H = \PP U$, where $U$ is an $s$-dimensional isotropic subspace, and $\dim(W \cap F_{\lambda_i}) = i$ for $1 \leq i < s$. Then, arguing as in the previous cases, $Y$ is irreducible and represents the class $m \sigma$  (see \cite{coskun:symplectic}).
\end{proof}

\subsection{Construction 3: Correspondences.} Correspondences give another way of constructing irreducible subvarieties. Consider the following diagram.
\begin{center}
\setlength{\unitlength}{5pt}
\begin{picture}(27,12)
\put(2,2){$X$} 
\put(8,7){\vector(-1,-1){3}}
\put(9,9){$Z$}
\put(12,7){\vector(1,-1){3}}
\put(16,2){$Y$}
\put(4,6){$\eta$}
\put(14,6){$\tau$}
\end{picture}
\label{f:correspondence} \\
\textsc{Figure 3}.  A correspondence between $X$ and $Y$.
\end{center}
Suppose that $X,Y$ and $Z$ are irreducible, projective varieties and that the fibers of $\tau$ are irreducible of the same dimension. Then the theorem on the dimension of fibers \cite[I.6.8]{shafarevich} implies that given an irreducible subvariety $W \subset Y$, $\tau^{-1}(W)$ is an irreducible subvariety of $Z$. Since the image of an irreducible variety is irreducible, $\eta(\tau^{-1}(W))$ (with its reduced induced structure) is an irreducible subvariety of $X$. 

Let $U$ and $W$ be irreducible subvarieties of $Y$ such that $[W]= m [U]$. If $\eta$ is generically injective on both $\tau^{-1}(W)$ and $\tau^{-1}(U)$, then $[\eta(\tau^{-1}(W))] = m[\eta(\tau^{-1}(U))]$.  This may be seen as follows. Since $[W]= m [U]$, by flat pull-back \cite[1.7.1]{fulton} $$[\tau^{-1}(W)]= \tau^*([W])= \tau^*(m[U]) = m \tau^*([U])= m [\tau^{-1}(U)].$$ Since $\eta$ is generically injective on $\tau^{-1}(W)$ and $\tau^{-1}(U)$, then by proper push-forward \cite[1.4]{fulton} $$[\eta(\tau^{-1}(W))]= \eta_* \tau^*([W]) = m \eta_* \tau^*([U]) = m [\eta(\tau^{-1}(U))].$$ We thus obtain the following lemma.

\begin{lemma}\label{lemma-correspondence}
Consider the correspondence in Figure 3. Assume that $X,Y,Z$ are irreducible, projective varieties and $\tau$ is a flat morphism with irreducible fibers. Let $W$ and $U$ be irreducible subvarieties of $Y$ such that $[W] = m [U]$. If $\eta$ is generically injective on both $\tau^{-1}(W)$ and $\tau^{-1}(U)$, then $\eta(\tau^{-1}(W))$ is an irreducible subvariety that represents $m [\eta(\tau^{-1}(U))]$.
\end{lemma}

There are several special cases that are useful in constructing irreducible subvarieties of cominuscule varieties. First, take $Y=Z$, and $\tau$ equal to the identity and $\eta$ to be an embedding. Let $\sigma = [U]$ be the class of a subvariety in $Y$ and assume that an irreducible subvariety $W \subset Y$ represents $m \sigma$ in $Y$. Then $W \subset X$ represents  $m[U]$ in the cohomology of $X$. These two classes are certainly cohomologous in $X$ since they are already cohomologous in $Y$. We obtain the following corollary. 

\begin{corollary}\label{stars}
Let $\sigma$ be one of the Schubert classes $(6: 3)$, $(8: 7)$ in $\mathbb{O}\PP^2$ or $(7 : 3)$, $(9:7)$, $(10:9)$, $(11:21)$, $(13:45)$ in $G(\mathbb{O}^3, \mathbb{O}^6)$.  (These classes correspond to the nodes marked by $*$ in Figures \ref{f:E6} and \ref{f:E7}.)  Then $m \sigma$ can be represented by an irreducible subvariety.
\end{corollary}

\begin{proof}
The Schubert variety corresponding to the node marked by $C(OG(5,10))$ in Figure \ref{f:E6} is a cone over the spinor variety $OG(5,10)$, cf. Example \ref{eg:E6b}.  
The Schubert variety $\Sigma_{2,3} \subset OG(5,10)$ is the unique Schubert variety of dimension 5 and degree 3, and $\Sigma_{2} \subset OG(5,10)$ is the unique Schubert variety of dimension 7 and degree 7.  Both the classes $\sigma_{2,3}$ and $\sigma_{2}$ in $OG(5,10)$ are obstructed (Section \ref{S:Spin}).
By Corollary \ref{cor-isotropic}, there exists irreducible subvarieties $Y$ representing the classes $m \sigma_{2,3}$ and  $m \sigma_{2}$ in $OG(5,10)$.
Taking the cone over $Y$, we obtain irreducible subvarieties of $\mathbb{O} \PP^2$ representing $m$ times the Schubert classes $(6:3)$ and $(8:7)$, see Figure \ref{f:E6}.

The variety corresponding to the node marked by $C(\mathbb{O}\PP^2)$ in Figure \ref{f:E7} is the cone over the Cayley plane $\mathbb{O} \PP^2$, cf. Example \ref{eg:E7b}.  By the previous paragraph and Corollary \ref{E6}, multiples of the Schubert classes $(6 : 3)$, $(8:7)$, $(9:9)$, $(10:21)$, $(12:45)$ in $\mathbb{O}\PP^2$ can be represented by irreducible subvarieties of $\mathbb{O}\PP^2$. Taking the cone over these varieties, we conclude that multiples of the classes $(7 : 3)$, $(9:7)$, $(10:9)$, $(11:21)$, $(13:45)$ in $G(\mathbb{O}^3, \mathbb{O}^6)$ can be represented by irreducible subvarieties, see Figure \ref{f:E7}.
\end{proof}

Next, we take Figure 3 to be a Tits correspondence (Section \ref{S:tits}) by setting $X= G/P$, $Y= G/Q$, $Z= G/(P \cap Q)$ and $\eta$ and $\tau$ to be the natural projections.  For a subvariety $W \subset Y$, let $\mathcal{T}(W) = \eta(\tau^{-1}(W))$ denote the Tits transform.  Let $d_{\tau} = \dim(\tau^{-1}(p))$ for a point $p \in G/Q$.

\begin{lemma}\label{lemma-injectivity}
Let $\Sigma$ be a Schubert variety in $G/Q$. If $\dim(\mathcal{T}(\Sigma)) = \dim(\Sigma) + d_{\tau}$, then $\eta$ is generically injective on $\tau^{-1}(\Sigma)$.
\end{lemma}

\begin{proof}
As noted in Section \ref{S:tits}, the Tits transform $\mathcal{T}(\Sigma)$ is a Schubert variety in $G/P$.  
Let $\Gamma$ be a general Schubert variety Poincar\'{e} dual to $\mathcal{T}(\Sigma)$. Since $[\Gamma] \cdot [\mathcal{T}(\Sigma)]=1$, $\Gamma$ and $\mathcal{T}(\Sigma)$ intersect at a unique point $p$. 
Note that $\dim(\tau^{-1}(\Sigma)) = \dim(\Sigma) + d_\tau$.  Therefore, $\dim(\mathcal{T}(\Sigma)) = \dim(\tau^{-1}(\Sigma))$.  It follows 
that the Schubert variety $\eta^{-1}(\Gamma)$  is  of complementary dimension to $\tau^{-1}(\Sigma)$ in $G/(P\cap Q)$. If the two are Poincar\'{e} dual in $G/(P\cap Q)$, then $[\eta^{-1}(\Gamma)]\cdot [\tau^{-1}(\Sigma)]=1$ and there is a unique point of $\tau^{-1}(\Sigma)$ lying over $p$. Therefore, $\eta$ is generically injective on $\tau^{-1}(\Sigma)$. Otherwise, $[\eta^{-1}(\Gamma)]\cdot [\tau^{-1}(\Sigma)]=0$. Since the two varieties intersect, they cannot intersect dimensionally properly. We conclude that $\tau^{-1}(\Sigma)$ has a positive dimensional fiber over $p$. Since $p$ is general, we conclude that the fibers of $\eta$ over $\mathcal{T}(\Sigma)$ are positive dimensional. Hence, by the theorem on the dimension of fibers, $\dim(\mathcal{T}(\Sigma)) < \dim(\Sigma) + d_{\tau} = \dim(\tau^{-1}(\Sigma))$. This contradicts our assumption. We conclude that $\eta$ is generically injective on $\tau^{-1}(\Sigma)$.
\end{proof}

\begin{corollary}\label{cor-prop25}
Suppose that $\Sigma \subset \Sigma'$ is a Schubert divisor. Suppose that there is an irreducible subvariety $W$ representing $m [\Sigma]$ contained in $\Sigma'$ and passing through a general point of $\Sigma'$. If $\dim(\mathcal{T}(\Sigma')) = \dim(\Sigma') + d_{\tau}$, then $\mathcal{T}(W)$ represents $m[\mathcal{T}(\Sigma)]$. 
\end{corollary}

\begin{proof}
By Lemma \ref{lemma-injectivity}, $\eta$ is generically injective on $\tau^{-1}(\Sigma')$. Since both $W$ and $\Sigma$ pass through a general point of $\Sigma'$, $\eta$ is also injective on $\tau^{-1}(W)$ and $\tau^{-1}(\Sigma)$. Lemma \ref{lemma-correspondence} implies the corollary.
\end{proof}

\begin{corollary}\label{E6}
Let $\sigma$ be one of the Schubert classes $(9:9), (10:21)$ or $(12:45)$ marked by $T$ in the Hasse diagram of $\OOO \PP^2$. Then every positive multiple of $\sigma$ can be represented by an irreducible subvariety.
\end{corollary}

\begin{proof}
The Cayley plane  is uniruled by smooth quadrics $\mathcal{Q}^8$ of dimension 8 and the ruling is parameterized by the dual Cayley plane $E_6/P_1$, see Example \ref{eg:E6a}.b. The dual Cayley plane contains a family of maximal $\PP^4$'s (corresponding to the $\PP^4$ denoted by the rigid node in Figure \ref{f:E6}).   Applying Lemma \ref{L:tits}, we see the Tits transforms of the linear spaces $\PP^0 \subset \PP^1 \subset \PP^2 \subset \PP^3 \subset \PP^4 \subset E_6/P_1$ are the Schubert varieties $(8:2) \subset (9:9) \subset (10:21) \subset (11:33) \subset (12:45)$, respectively. Here $d_\tau=8$, so by Corollary \ref{cor-linear} and Corollary \ref{cor-prop25}, we conclude that every positive multiple of the Schubert classes $(9:9), (10:21)$ and $(11:33)$ can be represented by an irreducible subvariety. For example, the Schubert varieties $(9:9)$ and $(10:21)$ are swept out by a pencil and a net of smooth 8-dimensional quadrics, respectively. The families of quadrics parameterized by a plane curve of degree $m$ or a surface in $\PP^3$ of degree $m$  define irreducible subvarieties of $\OOO \PP^2$ that represent $m (9:9)$ and $m (10:21)$, respectively. 

The Cayley plane is also uniruled by lines parameterized by $E_6/P_5$, see Example \ref{eg:E6a}.a. 
Given a Schubert variety $\Sigma \subset E_6/P_5$, let $\mathcal{T}(\Sigma)$ denote the Tits transform to $E_6/P_6 = \mathbb{O}\PP^2$; given $\Sigma' \subset \mathbb{O}\PP^2$, let $\mathcal{T}'(\Sigma')$ denote the Tits transform to $E_6/P_5$.  Then, by Lemma \ref{L:tits}, applying $\mathcal{T}\circ\mathcal{T}'$ to the linear spaces $\PP^0 \subset \PP^1 \subset \PP^2$ in $\OOO \PP^2$, we obtain the Schubert varieties with classes $(11:12) \subset (12:45) \subset (13:78)$ in $\OOO \PP^2$. Since $\mathbb{O}\PP^2 = E_6/P_6$, $E_6/P_5$ and $E_6/P_{5,6}$ are of dimensions $16$, $25$ and $26$, respectively, we see that $d_\tau = 1$ and $d_{\tau'} = 10$.  
It follows from Corollary \ref{cor-linear} and Corollary \ref{cor-prop25} that every positive multiple of the class $(12:45)$ can be represented by an irreducible subvariety. More concretely, the Schubert variety with class $(12:45)$ parameterizes lines in $\OOO \PP^2$ that intersect a fixed line. By considering lines that intersect a plane curve of degree $m$, we obtain irreducible representatives of the class $m (12:45)$. This concludes the proof of the corollary.
\end{proof}

\begin{corollary}\label{E7}
If $\sigma$ is a Schubert class in the Hasse diagram of the Freudenthal variety marked by a T, then every positive multiple of $\sigma$ can be represented by an irreducible subvariety.
\end{corollary}
\begin{proof}
The Freudenthal variety is ruled by quadrics $\mathcal{Q}^{10}$ of dimension 10 parameterized by $E_7/P_1$, see Example \ref{eg:E7a}.b.  By the same token, the variety $E_7/P_1$ is uniruled by Cayley planes parameterized by $E_7/P_7$. In particular, the Cayley plane occurs as a Schubert variety in $E_7/P_1$ (as the Tits transform of a point in $G(\OOO^3, \OOO^6)$). Table \ref{t:tits} summarizes the Tits transforms of Schubert cycles contained in $\OOO \PP^2 \subset E_7/P_1$. The rigid Schubert classes have been indicated by a $*$.  Since $E_7/P_1$ and $E_7/P_{1,7}$ have dimensions $33$ and $43$, respectively, we see that $d_\tau = 10$.   By Corollary \ref{cor-prop25}, it follows that all the positive multiples of the classes on the Hasse diagram marked by a T, except for $(18:520)$, can be represented by an irreducible subvariety since the corresponding classes in $\OOO \PP^2$ can be represented by irreducible subvarieties.

To see that the multiples of the  class $(18:520)$ can be represented by irreducible subvarieties, we consider the ruling of $G(\OOO^3, \OOO^6)$  by lines parameterized by $E_7/P_6$ (Example \ref{eg:E7a}.a). Given a Schubert variety $\Sigma \subset E_7/P_6$, let $\mathcal{T}(\Sigma)$ denote the Tits transform to $E_7/P_7 = \freud$; given $\Sigma' \subset \freud$, let $\mathcal{T}'(\Sigma')$ denote the Tits transform to $E_7/P_6$.  Then, by Lemma \ref{L:tits}, applying $\mathcal{T}\circ\mathcal{T}'$ to the linear spaces $\PP^0 \subset \PP^1 \subset \PP^2$ in $G(\OOO^3, \OOO^6)$, we obtain the Schubert varieties $(17:78) \subset (18:520) \subset (19:1710)$ in $G(\OOO^3, \OOO^6)$.  Since $\freud = E_7/P_7$, $E_7/P_6$ and $E_7/P_{6,7}$ have dimensions $27$, $42$ and $43$, respectively, we see that $d_\tau = 1$ and $d_{\tau'} = 16$. By Corollary \ref{cor-prop25}, it follows that every positive multiple of $(18:520)$ can be represented by an irreducible subvariety. More concretely, the Schubert variety with class $(18:520)$ parameterizes lines in $G(\OOO^3, \OOO^6)$ that intersect a fixed line. By considering the variety of lines that intersect a plane curve of degree $m$, we obtain irreducible representatives of $m (18:520)$.  This concludes the proof of the corollary.
\end{proof}

We are now ready to complete the proof of Theorem \ref{main-theorem}.
\begin{proof}[Proof of Theorem \ref{main-theorem}]
Let $X$ be a cominuscule homogeneous variety and let $\sigma$ be an obstructed Schubert class in the cohomology of $X$. If $X$ is a Grassmannian $G(k,n)$, then by Proposition \ref{prop-Grassmannian}, every positive multiple of $\sigma$ can be represented by an irreducible subvariety. If $X$ is a quadric hypersurface, then by Corollary \ref{cor-quadric}, every positive multiple of $\sigma$ can be represented by an irreducible subvariety. If $X$ is the Lagrangian Grassmannian $LG(n,2n)$ or the orthogonal Grassmannian $OG(n,2n)$, then by Corollary \ref{cor-isotropic}, every positive multiple of $\sigma$ can be represented by an irreducible subvariety. Finally, if $X$ is the Cayley plane or the Freudenthal variety, observe that every obstructed class (those nodes in Figures \ref{f:E6} and \ref{f:E7} that are not encircled) is decorated by one of $+, *$ or T.  By Corollary \ref{cor-Bertini-exceptional}, Corollary \ref{stars} and Corollaries \ref{E6} and \ref{E7}, respectively, multiples of $\sigma$ can be represented by irreducible subvarieties. This concludes the proof of Theorem \ref{main-theorem}.
\end{proof}





\end{document}